\documentclass[11pt]{article}

\pdfoutput=1
\usepackage[pdftex]{color}
\usepackage{amssymb}
\usepackage{amsthm}
\usepackage{amsmath}
\usepackage{latexsym}
\usepackage{amscd}
\usepackage{graphicx}
\usepackage[pdftex, colorlinks=true, citecolor=green]{hyperref}
\usepackage{lscape}
\usepackage{multirow}
\usepackage{enumitem}
\usepackage{lineno}
\usepackage[normalem]{ulem}

 \usepackage{afterpage} %Please note that we need this package in order to have the pictures in the right place
  \usepackage{float} 
 \usepackage{wrapfig}
\usepackage{lscape}
\usepackage[figuresright]{rotating}

\definecolor{darkgreen}{rgb}{0,0.6,0}
\definecolor{darkblue}{rgb}{0.0,0.0,0.8}
\definecolor{orange}{rgb}{1.0,0.45,0.0}
\definecolor{teal}{rgb}{0.12,0.73,0.85}

\newcommand{\ds}{\displaystyle}

\newcommand{\ben}{\begin{equation}}     %equation
\newcommand{\eeqn}{\end{equation}}
\newcommand{\bey}{\begin{eqnarray}}
\newcommand{\eey}{\end{eqnarray}}

\newtheorem{thm}{Theorem}[section]

\newtheorem{lemma}[thm]{Lemma}

\begin{document}

\begin{flushleft}
{\Large
\textbf{Mandelbrot sets for fixed template iterations}
}
\\
\vspace{4mm}
Mark Comerford$^{1,}$\footnote[1]{Associate Professor, Department of Mathematics, University of Rhode Island, 5 Lippitt Road, Kingston, RI 02881; Email: mcomerford@math.uri.edu}, Anca R\v{a}dulescu$^2,$\footnote[2]{Associate Professor, Department of Mathematics, SUNY New Paltz, 1 Hawk Road, New Paltz, NY 12561; Email: radulesa@newpaltz}, Kieran Cavanagh$^3$
\\

\indent $^1$ \emph{Department of Mathematics, University of Rhode Island}

\indent $^2$ \emph{Department of Mathematics, SUNY New Paltz}

\indent $^3$ \emph{Department of Mathematics, Penn State University}

\end{flushleft}

\vspace{3mm}
\begin{abstract}
We study the dynamics of template iterations, consisting of arbitrary compositions of functions chosen from a finite set of polynomials. In particular, we focus on templates using complex unicritical maps in the family $\{ z^d + c, c \in \mathbb{C}, d \ge 2 \}$. We examine the dependence on parameters of the connectedness locus for a fixed template and show that, for most templates, the connectedness locus moves upper semicontiuously. On the other hand, one does not in general have lower semicontinuous dependence, and we show this by means of a counterexample.
\end{abstract}

\section{Introduction on template iterations} %\textcolor[rgb]{1,0,0.8}{Capital Letters?}}

In previous work by the second author, we first introduced questions concerning the asymptotic dynamics of template iterations of two quadratic maps from the family $f_c : \mathbb{C} \to \mathbb{C}, f_{c}(z) = z^2+c$ where $c \in \mathbb{C}$ is a constant parameter~\cite{JNLS,radulescu2016symbolic}. This can be viewed as a particular case of the non-autonomous iteration of polynomials, which was first studied in 1991 by Fornaess and Sibony~\cite{fornaess1991random}. Additional work was done 
done by Br\"uck et al \cite{BR1, BBR,BR2}, by Sumi \cite{Sumi1, Sumi2, Sumi3, sumi2013dynamics, Sumi2Generator}, and by the first author \cite{Com3, comerford2006hyperbolic,Com5,CW1, CW2}. 

The framework of template iterated systems can be generalized to apply to any number of generating monic unicritical polynomials of the form $f_{(c,d)} = z^d + c$. More precisely, 
given any natural number $D$, any $D$-tuple ${\bf c} = (c_0,  \ldots \ldots, c_{D-1}) \in \mathbb{C}^D$, any $D$-tuple of integers ${\bf d} = (d_0,  \ldots \ldots, d_{D-1})$ with $d_i \ge 2$ for each $0 \le i \le D-1$, and any sequence ${\bf s} = \{s_m\}_{m=1}^\infty \in \{ 0,\ldots \ldots, D-1 \}^\infty$ (which we also call a \emph{template}), we consider the iterated system in which the polynomial $P^{{\bf s},{\bf c},{\bf d}}_m$ iterated at each step $m$ is $f_{(c_i, d_i)}=z^{d_i} + c_i$, if $s_m=i$. We then say a sequence $\{P_m^{{\bf s},{\bf c},{\bf d}} \}_{m=1}^\infty$ of monic centered unicritical polynomials is \emph{generated by the $D$-tuples ${\bf c} = (c_0,  \ldots \ldots, c_{D-1})$,  ${\bf d}= (d_0,  \ldots \ldots, d_{D-1})$} and the template ${\bf s}$ if it is constructed in this way. For two fixed iterative times $0 \le m < n$, we will also need to consider the polynomial composition taken from time $m$ to time $n$ given by
$$Q^{{\bf s},{\bf c},{\bf d}}_{m,n} = f_{(c_{s_n},d_{s_n})} \circ \cdots \cdots \circ f_{(c_{s_{m+1}}, d_{s_{m+1}})} = P^{{\bf s},{\bf c},{\bf d}}_n \circ \cdots \cdots \circ P^{{\bf s},{\bf c},{\bf d}}_{m+1}.$$

\noindent For convenience we will denote in particular $Q_{0,n}^{{\bf s},{\bf c},{\bf d}}$ by $Q_n^{{\bf s},{\bf c},{\bf d}}$ and we will refer to the $D$-tuples ${\bf c}$ and ${\bf d}$ as respectively the \emph{constant} and \emph{degree} vectors associated with our template iteration. If $D = 1$, then clearly we recover the classical situation of iteration with a single monic centered unicritical polynomial. Note that, in view of Theorem 2.1 in \cite{Com3}, any non-autonomous sequence of unicritical polynomials will be conjugate in the appropriate sense to a sequence of monic centered unicritical polynomials. However, in general, the members of such a monic centered sequence will not be chosen from among finitely many possibilities as we consider here (even when the members of the original sequence were chosen from among finitely many possible polynomials). Finally, we remark that template iteration using polynomials of varying degrees has already appeared in the literature. See for example Theorem 1.12 in \cite{CW2} where the authors consider all possible template sequences arising from choices among the two polynomials $z^2-1$ and $z^3$. 

%For a fixed $d$-tuple of points {\color{orange}${\bf c} \in \mathbb{C}^d$ and a fixed template ${\bf s} \in \{0,\ldots \ldots, d-1\}^\infty$, we can consider the associated \emph{template iterated system} where} the orbit of any $\xi_0 \in \mathbb{C}$ is the sequence $o^{{\bf s},{\bf c}}(\xi_0) = (\xi_m)_{m \geq 0}$ constructed recursively, for every $m \geq 1$, as
%$$\xi_{m} = f_{c_{s_{m}}} (\xi_{m-1}) = Q^{{\bf s},{\bf c}}_m(\xi_0).$$ 

%%%%%%%%%%

Let $\sigma$ denote the shift map on the template space $\{ 0,\ldots \ldots, D-1 \}^\infty$ and for $m \ge 0$ let $\sigma^{\circ m}$ denote the $m$th iterate of $\sigma$, which has the effect of truncating the first $m$ members of any given  sequence (here, we just take $\sigma^{\circ 0}$ to be the identity).
For any template iterated system, since the leading terms of the finitely many polynomials $f_{(c_i, d_i)}(z)$, $0 \le i \le D-1$ dominate when $|z|$ is large, one can then consider, in a similar fashion as in the traditional case of iterations with a single map, the \emph{iterated basin of infinity at time $m$}, arising from the truncated sequence $\sigma^{\circ m}({\bf s})$, on which all points escape locally uniformly to infinity under the compositions $Q^{{\bf s},{\bf c},{\bf d}}_{m,n}$ as $n$ tends to infinity. We denote this basin of infinity by ${\cal A}^{\sigma^{\circ m}({\bf s}),{\bf c},{\bf d}}_\infty$ or more simply ${\cal A}^{{\bf s},{\bf c},{\bf d}}_{\infty,m}$.

The complement of ${\cal A}^{{\bf s}, {\bf c}, {\bf d}}_{\infty,m}$, for each $m \geq 0$, is then the set of points whose orbits do not escape to infinity (which is easily seen to be equivalent to the condition that the orbits be bounded) and is called the \emph{iterated filled Julia set at time $m$}. We denote this set by ${\cal K}^{\sigma^{\circ m}({\bf s}),{\bf c},{\bf d}} $ or more simply ${\cal K}^{{\bf s},{\bf c},{\bf d}}_m$:
$${\cal K}^{\sigma^{\circ m}({\bf s}),{\bf c},{\bf d}} = {\cal K}^{{\bf s},{\bf c},{\bf d}}_m = \{ z \in \mathbb{C}; \: Q^{{\bf s},{\bf c},{\bf d}}_{m,n}(z) \text{ is bounded}\}.$$ 

Finally, for each $m \ge 0$ we have the \emph{iterated Julia set at time $m$}, which is simply the common boundary of the iterated filled Julia set and basin of infinity and is denoted by ${\cal J}^{\sigma^{\circ m}({\bf s}),{\bf c},{\bf d}} $ or just ${\cal J}^{{\bf s},{\bf c},{\bf d}}_m$ i.e. 
$${\cal J}^{\sigma^{\circ m}({\bf s}),{\bf c},{\bf d}} = {\cal J}^{{\bf s},{\bf c},{\bf d}}_m = \partial {\cal K}^{{\bf s},{\bf c},{\bf d}}_m.$$

The iterated Julia sets coincide with the iterated Julia sets from standard non-autonomous polynomial iteration and are precisely those points on which the compositions $Q^{{\bf s},{\bf c},{\bf d}}_{m,n}$ fail to give a normal family on any neighbourhood (see e.g. \cite{comerford2006hyperbolic} for details). 

When $m = 0$ we will refer to the corresponding iterated basin of infinity, filled Julia set, and Julia set as simply the \emph{basin of infinity}, \emph{filled Julia set}, and \emph{Julia set} respectively. These then 
clearly extend the corresponding definitions from standard polynomial iteration and for simplicity we usually denote them by ${\cal A}^{{\bf s},{\bf c},{\bf d}}_{\infty}$, ${\cal K}^{{\bf s},{\bf c},{\bf d}}$, and ${\cal J}^{{\bf s},{\bf c},{\bf d}}$, i.e. 

$${\cal A}^{{\bf s},{\bf c},{\bf d}}_{\infty} = {\cal A}^{{\bf s},{\bf c},{\bf d}}_{\infty,0},\; {\cal K}^{{\bf s},{\bf c},{\bf d}}= {\cal K}^{{\bf s},{\bf c},{\bf d}}_0,\; {\cal J}^{{\bf s},{\bf c},{\bf d}} = {\cal J}^{{\bf s},{\bf c},{\bf d}}_0.$$

When the $D$-tuples ${\bf c}$, ${\bf d}$ are fixed and there is no danger of confusion, we will omit ${\bf c}$, ${\bf d}$ from the superscript for simplicity, and use the notation $Q^{\bf s}_{m,n}$ for the composition maps, ${\cal A}^{\bf s}_{\infty,m}$ for the basin of infinity at time $m$, and respectively ${\cal K}^{\bf s}_m$ and ${\cal J}^{\bf s}_m$ for the iterated filled Julia sets and the iterated Julia sets at time $m \geq 0$ (with similar notation for the basin of infinity, filled Julia set, and Julia set at time $0$).

The term `iterated' for the iterated basins of infinity, filled Julia sets, and Julia sets is justified in view of the following invariance result, the proof of which is a relatively easy exercise we leave to the reader (also see~\cite{comerford2006hyperbolic} for a slightly stronger invariance statement).

\begin{lemma}[Complete Invariance] 
\label{completelyinvariant}
For any ${\bf s} \in \{ 0, \ldots \ldots, D-1 \}^\infty$ and $0 \le m  < n$, we have that $Q^{\bf s}_{m,n}({\cal A}^{\bf s}_{\infty, m}) = {\cal A}^{\bf s}_{\infty, n}$,  $Q^{\bf s}_{m,n}({\cal K}^{\bf s}_m) = {\cal K}^{\bf s}_n$, and $Q^{\bf s}_{m,n}({\cal J}^{\bf s}_m) = {\cal J}^{\bf s}_n$.
\end{lemma}

%%%%%%%%%%
%\noindent In the full parameter space $\mathbb{C}^2 \times \{ 0,1 \}^\infty$ for template iterated systems, we defined the \emph{Mandelbrot set} as the parameter locus for which the iteration is postcritically bounded, that is:
%$${\cal M} = \{ (c_0,c_1,{\bf s}) \in \mathbb{C}^2 \times \{ 0,1 \}^\infty, \text{ such that } o^{{\sigma^{\circ m}({\bf s})},(c_0,c_1)}(0) \text{ is bounded for all $m \geq 0$} \}.$$

% In other words, we require that the orbit initiated at zero at any point of the iteration be bounded. 

 For a fixed template ${\bf s} \in \{ 0, \ldots \ldots, D-1 \}^\infty$ and an upper bound $d \ge \max_{0 \le i \le D-1} d_i$ on degrees (which we introduce as we require compactness), we define the \emph{fixed-template Mandelbrot set} ${\cal M}^{\bf s}$ as a subset of $\mathbb{C}^D \times \{2, \ldots \ldots, d\}^D$ defined by:
$${\cal M}^{\bf s} = \{ ({\bf c}, {\bf d}) \in \mathbb{C}^D\times \{2, \ldots \ldots, d\}^D: Q^{{\bf s},{\bf c}, {\bf d}}_{m,n}(0) \text{ is bounded for all} \: 0 \le m \le n \}.$$

%Lastly, for a fixed $(c_0,c_1) \in \mathbb{C}^2$, we defined the \emph{fixed-map Mandelbrot slice} as a subset of $\{ 0,1 \}^\infty$:
%$${\cal M}_{c_0,c_1} = \{ {\bf s} \in \{ 0,1 \}^\infty, \text{ such that } (c_0,c_1,{\bf s}) \in {\cal M} \}.$$

%%%

\noindent{\bf Remarks.} In the case of template iteration where the polynomials have constant degree, the fixed-template Mandelbrot set ${\cal M}^{\bf s}$ can obviously be identified with a subset of $\mathbb{C}^D$ as is more common in classical polynomial iteration (and as we will do ourselves in Section 3). Since having an unbounded orbit is clearly an open condition, it is easy to see that ${\cal M}^{\bf s}$ is automatically a closed subset of $\mathbb{C}^D\times \{2, \ldots \ldots, d\}^D$ (where we endow $\{2, \ldots \ldots, d\}^D$ with the discrete topology).\\

With these extensions, some of the results established in the traditional context of single map iterations are easily generalizable for template iterations. Others are nontrivial to extend, or do not hold true for templates. We already remarked earlier that template iterations is a special case of the non-autonomous iteration of polynomials. We  also note that there is a strong connection between template iteration and polynomial semigroups as studied by Sumi and  Stankewitz ~\cite{stankewitz2011dynamical, Sumi1, Sumi2,  Sumi2006, Sumi3, sumi2013dynamics}, especially for finitely generated polynomial semigroups as considered in \cite{ Sumi2Generator}.

Some results for templates follow trivially from the general theory of non-autonomous iteration, but we also have have objects and properties which apply only to the context of templates. For example, we showed in prior work that template iterated systems have an escape radius, and that, when $D = 2$ and $d_0 = d_1 = 2$, $M = \max\{2,|c_0|,|c_1|\}$ acts as such an escape radius for any template iteration of the two maps arising from the pair $(c_0,c_1)$~\cite{JNLS}. In this paper we want to refine this result and we show that the fixed template Mandelbrot set ${\cal M}^{\bf s}$ is a subset of the set $\overline{\mathrm{D}}(0,2)^D \times \{2, \ldots \ldots, d\}^D \subset \mathbb{C}^d \times \{2, \ldots \ldots, d\}^D$ for most templates ${\bf s}$.

For any ${\bf s} \in \{ 0, \ldots \ldots, D-1 \}^\infty$, it remains true that the fixed template Mandelbrot set coincides with the connectedness locus for the Julia set corresponding to the same template:

\vspace{.2cm}
\begin{lemma}
\label{CharacterizingM}
The fixed template Mandelbrot set ${\cal M}^{\bf s}$ is equal to each of the following sets: 
\begin{eqnarray*}
&&\{ ({\bf c}, {\bf d}) \in \mathbb{C}^D\times \{2, \ldots \ldots, d\}^D: {\cal K}^{{\bf s}, {\bf c}, {\bf d}} \text{ is connected } \},\\
&&\{ ({\bf c}, {\bf d}) \in \mathbb{C}^D\times \{2, \ldots \ldots, d\}^D: {\cal J}^{{\bf s}, {\bf c}, {\bf d}} \text{ is connected } \},\\
&& \{ ({\bf c}, {\bf d}) \in \mathbb{C}^D\times \{2, \ldots \ldots, d\}^D : {\cal K}^{{\bf s}, {\bf c}, {\bf d} }_m \text{ is connected, for all } m \geq 0 \},\\
&& \{ ({\bf c}, {\bf d}) \in \mathbb{C}^D\times \{2, \ldots \ldots, d\}^D : {\cal J}^{{\bf s}, {\bf c}, {\bf d} }_m \text{ is connected, for all } m \geq 0 \},\\
&& \{ ({\bf c}, {\bf d}) \in \mathbb{C}^D\times \{2, \ldots \ldots, d\}^D :  0 \in {\cal K}^{{\bf s}, {\bf c}, {\bf d} }_m \text{ for all } m \geq 0 \},\\
&& \{ ({\bf c}, {\bf d}) \in \mathbb{C}^D\times \{2, \ldots \ldots, d\}^D : c_{s_m} \in {\cal K}^{{\bf s}, {\bf c}, {\bf d} }_m \text{ for all } m \geq 1 \}.\\
\end{eqnarray*}

\end{lemma}

\vspace{-.4cm}
\noindent The proof carries through for non-autonomous (and in particular template) iterations in a similar way to that for single polynomial maps (using Green's functions where a critical point in the basin of infinity gives rise to a level curve which separates the Julia set \cite{BR2}). We will make use of this fact in the following sections.

%%%%%%%%

\section{Boundedness of the template Mandelbrot set}

It is easy to see from the standard theory of iterations with a single quadratic polynomial that, for some templates ${\bf s} = \{s_m\}_{m=1}^\infty \in \{ 0,\ldots \ldots, D-1 \}^\infty$ and degree bounds $d$, the fixed template Mandelbrot set ${\cal M}^{\bf s}$ is unbounded in $\mathbb{C}^D\times \{2, \ldots \ldots, d\}^D$. More precisely, if there exists $i$, $0 \le i \le D-1$ for which $s_m \ne i$ for all $m \ge 1$, so that $P^{{\bf s},{\bf c},{\bf d}}_m \ne z^{d_i} + c_i$ for all $m \ge 1$, then ${\cal M}^{\bf s}$ will be unbounded in the $i$-th coordinate as there is no constraint on the corresponding critical value $c_i$ in parameter space. With this in mind, we say that the template ${\bf s} = \{s_m\}_{m=1}^\infty$ is \emph{full} if this does not occur and that, for each $0 \le i \le d-1$, $s_m = i$ and thus $P^{{\bf s},{\bf c},{\bf d}}_m = z^{d_i} + c_i$ for at least one value of $m$. We observe that this condition is clearly generic. We will show that, in the case of a full template, ${\cal M}^{\bf s}$ is bounded as one might expect. 

We first start with a technical lemma which makes use of a well-known fact from potential theory. In this paper, for $z\in \mathbb C$ and $r > 0$, we use the common notation ${\mathrm D}(z,r)$, $\overline {\mathrm D}(z,r)$ to denote the open and closed discs respectively with centre $z$ and radius $r$.

%a constant template, the fixed template Mandelbrot set ${\cal M}^{\bf s}$ is unbounded in $\mathbb{C}^2$. More precisely, ${\cal M}^{\bf s} = {\cal M} \times \mathbb{C}$ if ${\bf s} = 000 \hdots$, while ${\cal M}^{\bf s} = \mathbb{C} \times {\cal M}$ if ${\bf s} = 111 \hdots$, where ${\cal M} \subset \mathbb{C}$ is the traditional Mandelbrot set for quadratic polynomials of the form $f_c(z) = z^2 + c$. On the other hand, as one might expect, one can indeed show that ${\cal M}^{\bf s}$ is bounded for any nonconstant template ${\bf s}$. We first start with a technical lemma which makes use of some ideas from potential theory. 
%\color{black}

\begin{lemma}
For any ${\bf s} \in \{ 0,\ldots \ldots D-1 \}^\infty$ and any pair of $d$-tuples $({\bf c}, {\bf d})  \in {\cal M}^{\bf s}$, the iterated filled Julia sets ${\cal K}^{{\bf s},{\bf c},{\bf d}}_m$ satisfy
$${\cal K}^{{\bf s},{\bf c},{\bf d}}_m \subset \overline{\mathrm{D}}(0,2), \quad \mbox{for all}\:\; m \geq 0.$$
\label{boundedness1}
\end{lemma}

\begin{proof} 
%{\color{blue}Since both the template ${\bf s}$ and the pair $(c_0,c_1) \in {\cal M}^{\bf s}$ remain fixed in  throughout this argument, we will use the notation ${\cal K}^m_{\bf s} = {\cal K}^{c_0,c_1}_{\sigma^{\circ m}({\bf s})}$.}
By Lemma \ref{CharacterizingM} above, since $({\bf c}, {\bf d})  \in {\cal M}^{\bf s}$, for each $m\geq 0$ the iterated filled Julia set ${\cal K}_m^{{\bf s},{\bf c},{\bf d}}$ is connected. By Theorem 5.3.2 (a) in~\cite{ransford1995potential}, we have that $\text{diam}({\cal K}^{{\bf s},{\bf c},{\bf d}}_m) \le 4\,\text{cap}({\cal K}^{{\bf s},{\bf c},{\bf d}}_m)$, where $\text{diam}({\cal K}^{{\bf s},{\bf c},{\bf d}}_m)$ and $\text{cap}({\cal K}^{{\bf s},{\bf c},{\bf d}}_m)$ are respectively the diameter and the logarithmic capacity of the set ${\cal K}^{{\bf s},{\bf c},{\bf d}}_m$. Furthermore, since the polynomials $f_{(c_{s_m},d_{s_m})}$ are monic, $\text{cap}({\cal K}^{{\bf s},{\bf c},{\bf d}}_m)=1$ (see Theorem 1.4(5) in~\cite{comerford2006hyperbolic} or Theorem 2.1 in~\cite{fornaess1991random}), whence $\text{diam}({\cal K}^{{\bf s},{\bf c},{\bf d}}_m) \leq 4$. 

Recall that the diameter of a regular $n$-gon with vertices on the unit circle is $2$ for $n$ even and $\sqrt{2 - 2\cos \left ( \tfrac{2\pi k}{2k+1} \right )} < 2$ for $n = 2k + 1$ odd. Since $P^{{\bf s},{\bf c},{\bf d}}_{m+1}(z) = f_{(c_{s_{m+1}}, d_{s_{m+1}})}(z)  = z^{d_{s_{m+1}}} + c_{s_{m+1}}$, if $d_{m+1}$ is even, then ${\cal K}^{{\bf s},{\bf c},{\bf d}}_m$ is symmetric about $0$ under $z \mapsto -z$ and it follows immediately that ${\cal K}^{{\bf s},{\bf c},{\bf d}}_m \subset \overline{\mathrm{D}}(0,2)$. On the other hand, when $d_{m+1}$ is odd, we have the slightly weaker condition that ${\cal K}^{{\bf s},{\bf c},{\bf d}}_m \subset \overline{\mathrm{D}}(0,r)$ where $r = 4/\sqrt{2 - 2\cos(2\pi k/(2k+1 ))} > 2$, the worst case being when $d_{m+1} = 3$ which gives $r = 4/\sqrt{3}$. Thus in any case we can say that for each $m \ge 0$ we must have ${\cal K}^{{\bf s},{\bf c},{\bf d}}_m \subset \overline{\mathrm{D}}(0,4/\sqrt{3})$ (regardless of whether $d_{m+1}$ is even or odd) and so, by Lemma \ref{CharacterizingM}, since $({\bf c}, {\bf d})  \in {\cal M}^{\bf s}$, we have that 
\begin{equation}
\label{cmbound}
|c_m| \le 4/\sqrt{3}, \quad m \ge 0.
\end{equation}

\noindent However, if $m \ge 0$, $|z| > 2$, $d_{m+1} \ge 3$, and $|c_{m+1}| \le 4/\sqrt{3}$, we have $|P^{{\bf s},{\bf c},{\bf d}}_{m+1}(z)|= |f_{(c_{s_{m+1}}, d_{s_{m+1}})}(z)| > 8 - 4/\sqrt{3} > 4$ from which it follows using \eqref{cmbound} above that $z$ escapes to infinity under iteration using the sequence $\{Q^{{\bf s},{\bf c},{\bf d}}_{m,n}\}_{n=m+1}^\infty$.  Using this, it follows that ${\cal K}^{{\bf s},{\bf c},{\bf d}}_m \subset \overline{\mathrm{D}}(0,2)$, when $d_{m+1}$ is odd also, which completes the proof.
\end{proof}

\begin{lemma} 
If $\mathbf{s} \in \{ 0,\ldots \ldots, D-1 \}^\infty$ is a full template, then $\mathcal{M}^{\mathbf{s}}$ is a compact subset of $\overline{\mathrm{D}}(0,2)^D \times \{2, \ldots \ldots, d\}^D$.
\label{boundedness2}
\end{lemma}

\begin{proof}
Suppose $({\bf c}, {\bf d}) = ((c_0,  \ldots \ldots, c_{D-1}), (d_0,  \ldots \ldots, d_{D-1})) \in \mathcal{M}^{\mathbf{s}}$. Then, by Lemma \ref{CharacterizingM} it follows that for each $m \ge 1$ the critical value $c_{s_m}$ of $P^{{\bf s},{\bf c},{\bf d}}_m= f_{(c_{s_m}, d_{s_m})}$ belongs to $\mathcal{K}^{{\bf s}, {\bf c}, {\bf d}}_m$. Since ${\bf s}$ is full, it follows that, for each $0 \le i \le D-1$, there exists $m_i \ge 1$ such that $s_{m_i} = i$ and $P^{{\bf s}, {\bf c}, {\bf d}}_{m_i} = z^{d_i} + c_i$, (for some $2 \le d_i \le d$) so that $c_i \in \mathcal{K}^{{\bf s},{\bf c},{\bf d}}_{m_i}$. On the other hand, from  Lemma~\ref{boundedness1}, it follows that $\mathcal{K}^{{\bf s},{\bf c}, {\bf d}}_{m_i+1} \subset \overline{\mathrm{D}}(0,2)$. Thus, $({\bf c}, {\bf d}) \in  \overline{\mathrm{D}}(0,2)^D \times \{2, \ldots \ldots, d\}^D$, and, since $\mathcal{M}^{\mathbf{s}}$ is closed (as remarked earlier),  while the set $\{2, \ldots \ldots, d\}^D$ is finite, $\mathcal{M}^{\mathbf{s}}$ is also compact, which concludes our proof.
\end{proof}

\section{Limit behavior for arbitrary templates}

On the space of templates $\{ 0,\ldots \ldots, D-1\}^\infty$ we consider the ultrametric defined for any two templates ${\bf s} = \{s_m\}_{m=1}^\infty$, $\tilde{\bf s} = \{\tilde{s}_m\}_{m=1}^\infty$ by
$$\| {\bf s} - \tilde{\bf s} \| = \sum_{k=1}^\infty \frac{|s_k-\tilde{s}_k|}{D^k}.$$

\noindent We also make use of the product metric on $\mathbb{C}^D\times \{2, \ldots \ldots, d\}^D$. To be specific, for two points $({\bf c}, {\bf d}) = ((c_0,  \ldots , c_{D-1}), (d_0,  \ldots , d_{D-1})), ({\bf \tilde{c}}, {\bf \tilde{d}}) = ((\tilde{c}_0,\hdots,\tilde{c}_{D-1}), (\tilde{d}_0,\hdots,\tilde{d}_{D-1})) \in \mathbb{C}^D \times \{2, \ldots, d\}^D$, we have
$$d(({\bf c}, {\bf d}),({\bf \tilde{c}}, {\bf \tilde{d}})) = \max \{ \max\{|\tilde{c}_i -c_i|, |\tilde{d}_i -d_i|\}, \: 0 \le i \le D-1 \}.$$

\noindent We then use the notation 
\begin{eqnarray*}
 {\mathrm {B}}(({\bf c}, {\bf d}), \delta) &=& {\mathrm {B}}(((c_0,  \ldots \ldots, c_{D-1}), (d_0,  \ldots \ldots, d_{D-1})), \delta)\\  &=& {\mathrm{D}}(c_0, \delta) \times \cdots \cdots  \times  {\mathrm{D}}(c_{D-1}, \delta) \times (d_0 - \delta, d_0 + \delta) \times \cdots \cdots  \times   (d_{D-1} - \delta, d_{D-1} + \delta)
 \end{eqnarray*}
to denote the open balls with respect to this metric.

\medskip
\noindent The distance between a point $({\bf c}, {\bf d}) \in \mathbb{C}^D\times \{2, \ldots \ldots, d\}^D$ and a set $B \subset  \mathbb{C}^D\times \{2, \ldots \ldots, d\}^D$ is defined as: 
$$d(({\bf c}, {\bf d}),B) = \inf \{d(({\bf c}, {\bf d}),({\bf \tilde{c}},{\bf \tilde{d}})), \: ({\bf \tilde{c}},{\bf \tilde{d}}) \in B \}.$$

\noindent For two sets $A,B \in \mathbb{C}^D\times \{2, \ldots \ldots, d\}^D$, the distance from $A$ to $B$ is defined as: 
$$d(A,B) = \sup \{ d({\bf c},B), \: {\bf c} \in A \}.$$

\noindent Finally, the Hausdorff distance between two sets $A,B \in \mathbb{C}^D\times \{2, \ldots \ldots, d\}^D$ is then given by:
$$d_H(A,B) = \max \{ d(A,B), d(B,A) \}.$$

%%%%%%

\begin{thm} Let ${\bf s} \in \{ 0,\ldots,D-1 \}^\infty$ be a full template and consider a sequence of templates $\left\{{\bf s}^N\right\}^\infty_{N=1}$ that converges to ${\bf s}$ as $N \to \infty$. Then 
$$ \lim_{N \to \infty} d \left( {\cal M}^{{\bf s}^N}, {\cal M}^{\bf s} \right) = 0.$$
\label{usc}
\end{thm}

\begin{proof} Since $\left\{{\bf s}^N\right\}^\infty_{N=1}$ converges to ${\bf s}$, for any finite $m_0 \ge 1$, there exists $N_0$ such that the first $m_0$ components of components of the templates ${\bf s}^N$ and ${\bf s}$ coincide for all $N \geq N_0$. Since ${\bf s}$ is full, from this we see that  
${\bf s}^N$ must be full for all sufficiently large $N$. It follows from Lemma~\ref{boundedness2} that ${\cal M}^{\bf s} \subset \overline{\mathrm{D}}(0,2)^D \times \{2, \ldots \ldots, d\}^D$ and likewise ${\cal M}^{{\bf s}^N} \subset \overline{\mathrm{D}}(0,2)^D \times \{2, \ldots \ldots, d\}^D$ (for all sufficiently large $N$). 

Suppose that $({\bf c}, {\bf d}) = ((c_0,  \ldots \ldots, c_{D-1}), (d_0,  \ldots \ldots, d_{D-1})) \in \overline{\mathrm{D}}(0,2)^D \times \{2, \ldots \ldots, d\}^D$. From above is not hard to see, in a similar manner to the classical case for a single quadratic polynomial $z^2 + c$,
that $R=2$ acts as an escape radius for the polynomial sequence arising from the $D$-tuples $({\bf c}, {\bf d})$ in conjunction with any template and in particular ${\bf s}^N$ and ${\bf s}$ (note that the existence of this escape radius for the case $D=d=2$ is Theorem 2.6 in~\cite{JNLS} while the general case is proved similarly). If in addition we assume that $({\bf c}, {\bf d}) \in (\overline{\mathrm{D}}(0,2)^D \times \{2, \ldots \ldots, d\}^D) \setminus {\cal M}^{\bf s}$, then at least one of the critical points of the template iteration under ${\bf s}$ escapes. More precisely, there exist $m_e \geq 0$, $t_e \geq 1$  such that
$$| Q^{\bf s}_{m_e,m_e + t_e}(0) | = |f_{(c_{s_{m_e+t_e}},d_{s_{m_e+t_e}})} \circ \cdots \cdots \circ f_{(c_{s_{m_e+1}},d_{s_{m_e+1}})}(0) | > 3$$
(where the polynomials $f_{(c_{s_m}, d_{s_m})}$, $m_e+1 \le m \le m_e+t_e$ are constructed using the template ${\bf s}$ and the parameter $({\bf c}, {\bf d})$).

{Again since $\left\{{\bf s}^N\right\}^\infty_{N=1}$ converges to ${\bf s}$}, there exists $N_0$ such that the first $m_e+t_e$ entries of the templates ${\bf s}^N$ and ${\bf s}$ coincide for all $N \geq N_0$. Hence in particular:
$$Q^{{\bf s}^N}_{m_e,m_e+t_e}(0) = Q^{\bf s}_{m_e,m_e+t_e}(0)$$
so that
\begin{equation}
| Q^{{\bf s}^N}_{m_e,t_e}(0) | > 3, \: \text{ for all } \: N \geq N_0.
\label{ineq1}
\end{equation}

\noindent For $(\tilde {\bf c}, \tilde {\bf d}) = ((\tilde{c}_0,  \ldots \ldots, \tilde{c}_{D-1}), (\tilde{d}_0,  \ldots \ldots, \tilde{d}_{D-1})) \in {\mathbb C}^D \times \{2, \ldots \ldots, d\}^D$, one can define, similarly to our earlier notation, the compositions 
$$\tilde{Q}^{{\bf s}^N}_{m_e,m_e+t_e} = \tilde{Q}^{\bf s}_{m_e,m_e+t_e}: = f_{(\tilde c_{s_{t_e+m_e}},\tilde d_{s_{t_e+m_e}})} \circ \cdots \cdots \circ f_{(\tilde c_{s_{m_e+1}},\tilde d_{s_{m_e+1}})}$$ 
(where the polynomials $f_{(\tilde{c}_{s_m},\tilde{d}_{s_m})}$, $m_e+1 \le m \le m_e+t_e$ are constructed using the template ${\bf s}$, the parameter $\tilde {\bf c}$ which is allowed to vary over ${\mathbb C}^D$, and the degree bound $d$ which we regard as fixed).

It is easy to see that these polynomials are jointly continuous in $z$, and in each $\tilde {c}_i$, $0 \le i \le d-1$. In addition, the coefficients for these polynomials (in $z$) are fixed polynomials in the parameters $\tilde{c}_i$ whose degrees in each variable are at most $d^{m_e+t_e-1}$. Further, the degrees of these coefficient polynomials will be locally constant as the parameter $({\bf c}, {\bf d})$ moves continuously. Finally, the degrees of these polynomials (again in $z$) are at most $d^{m_e+t_e}$ and thus uniformly bounded. Hence the partial derivatives in each $\tilde{c}_i$, $0 \le i \le d-1$ of these  polynomials are locally uniformly bounded as functions of these $D$ variables and there thus exists $\delta >0$ such that:
\begin{equation}
d(({\bf c}, {\bf d}),({\bf \tilde{c}}, {\bf \tilde{d}})) < \delta \: \Longrightarrow \: |\tilde{Q}^{{\bf s}^N}_{m_e,m_e+t_e}(0) - Q^{{\bf s}^N}_{m_e,m_e+t_e}(0)| <1.
\label{ineq2}
\end{equation}

\noindent From~\eqref{ineq1} and~\eqref{ineq2}, it follows that, for all $(\tilde {\bf c}, \tilde {\bf d}) \in \mathrm{B}(({\bf c}, {\bf d}), \delta)$ and $N \geq N_0$, we have:
$$|\tilde{Q}^{{\bf s}^N}_{m_e,m_e+t_e}(0)| \geq |Q^{{\bf s}^N}_{m_e,m_e+t_e}(0)| - |\tilde{Q}^{{\bf s}^N}_{m_e,m_e+t_e}(0) -Q^{{\bf s}^N}_{m_e,m_e+t_e}(0)| > 3-1 =2.$$

\noindent Since from above $R=2$ is an escape radius for any iterated template system with $({\bf c}, {\bf d}) \in\overline{\mathrm{D}}(0,2)^D \times \{2, \ldots \ldots, d\}^D$, 
it follows that $(\tilde {\bf c}, \tilde {\bf d}) \notin {\cal M}^{{\bf s}^N}$. Hence we have shown that, given any $({\bf c}, {\bf d})  \in (\overline{\mathrm{D}}(0,2)^D \times \{2, \ldots \ldots, d\}^D) \setminus {\cal M}^{{\bf s}}$, we can find $N_0 \in {\mathbb N}$ and  $\delta > 0$ such that the ball ${\mathrm {B}}(({\bf c}, {\bf d}),\delta)$ is contained in  $(\mathbb{C}^D \times \{2, \ldots \ldots, d\}^D )\setminus {\cal M}^{{\bf s}^N}$ for all $N \geq N_0$.\\

Now let $\varepsilon>0$ and set 
$$N_\varepsilon = \{ ({\bf c}, {\bf d}) \in \overline{\mathrm {D}}(0,2)^D\times \{2, \ldots \ldots, d\}^D:  d(({\bf c}, {\bf d}) ,{\cal M}^{\bf s}) < \varepsilon \}.$$

\noindent By Lemma \ref{boundedness2}, we have ${\cal M}^{\bf s} \subset \overline{\mathrm {D}}(0,2)^D\times \{2, \ldots \ldots, d\}^D$, so that ${\cal M}^{\bf s} \subset N_\varepsilon$. It follows from above that for any $D$-tuples $({\bf c}, {\bf d}) \in (\overline{\mathrm{D}}(0,2)^D\times \{2, \ldots \ldots, d\}^D)  \setminus N_\varepsilon$, one can make $\delta > 0$ sufficiently small and $N_0 > 0$ sufficiently large to guarantee that
$$B(({\bf c}, {\bf d}) ,\delta) \subset (\mathbb{C}^D \times \{2, \ldots \ldots, d\}^D )\setminus {\cal M}^{{\bf s}^N},\quad N \geq N_0.$$

\noindent Since $(\overline{\mathrm{D}}(0,2)^D\times \{2, \ldots \ldots, d\}^D) \setminus N_\varepsilon$ is compact (this is the point in the proof where the existence of the degree bound $d$ is essential), it follows that we can actually choose $N_0$ large enough such that
$$ (\overline{\mathrm{D}}(0,2)^D\times \{2, \ldots \ldots, d\}^D) \setminus N_\varepsilon \subset (\mathbb{C}^D\times \{2, \ldots \ldots, d\}^D) \setminus {\cal M}^{{\bf s}^N}, \: \text{ for all } N \geq N_0$$
and, since ${\cal M}^{{\bf s}^N} \subset \overline{\mathrm {D}}(0,2)^D\times \{2, \ldots \ldots, d\}^D$, again from Lemma \ref{boundedness2} above, we have in fact that
$$ \mathbb{C}^D\times \{2, \ldots \ldots, d\}^D  \setminus N_\varepsilon \subset (\mathbb{C}^D\times \{2, \ldots \ldots, d\}^D) \setminus {\cal M}^{{\bf s}^N}, \: \text{ for all } N \geq N_0.$$
Hence
$${\cal M}^{{\bf s}^N} \subset N_\varepsilon, \: \text{ for all } N \geq N_0.$$
In conclusion, since ${\cal M}^{{\bf s}^N}$ is compact in view of Lemma \ref{boundedness2} (which again requires the existence of the degree bound $d$), we have 
$$d({\cal M}^{{\bf s}^N},{\cal M}^{\bf s}) <\varepsilon, \: \text{ for all } N \geq N_0$$
which completes the proof. 
\end{proof}

\vspace{4mm}
Basically, the above theorem states that the template Mandelbrot sets depend upper semicontinuoulsy on the parameter $({\bf c}, {\bf d})$ (see Figure \ref{whole_set_hole} below for an illustration of how this looks in practice). We remark the similarity between this result and the upper semicontinuous dependence of the filled Julia set in non-autonomous polynomial iteration (Theorem 2.4 in \cite{comerford2006hyperbolic}) as well as the classical iteration with a single polynomial. 
We will next show next that that this statement does not remain true in the other direction. In other words, we do not have lower semicontinuous dependence, and thus we do not have continuity with respect to the Hausdorff topology. To be precise, we will show that:

\begin{figure}[h!]
\label{WholeSets}
\begin{center}
    \includegraphics[width=1.0\textwidth]{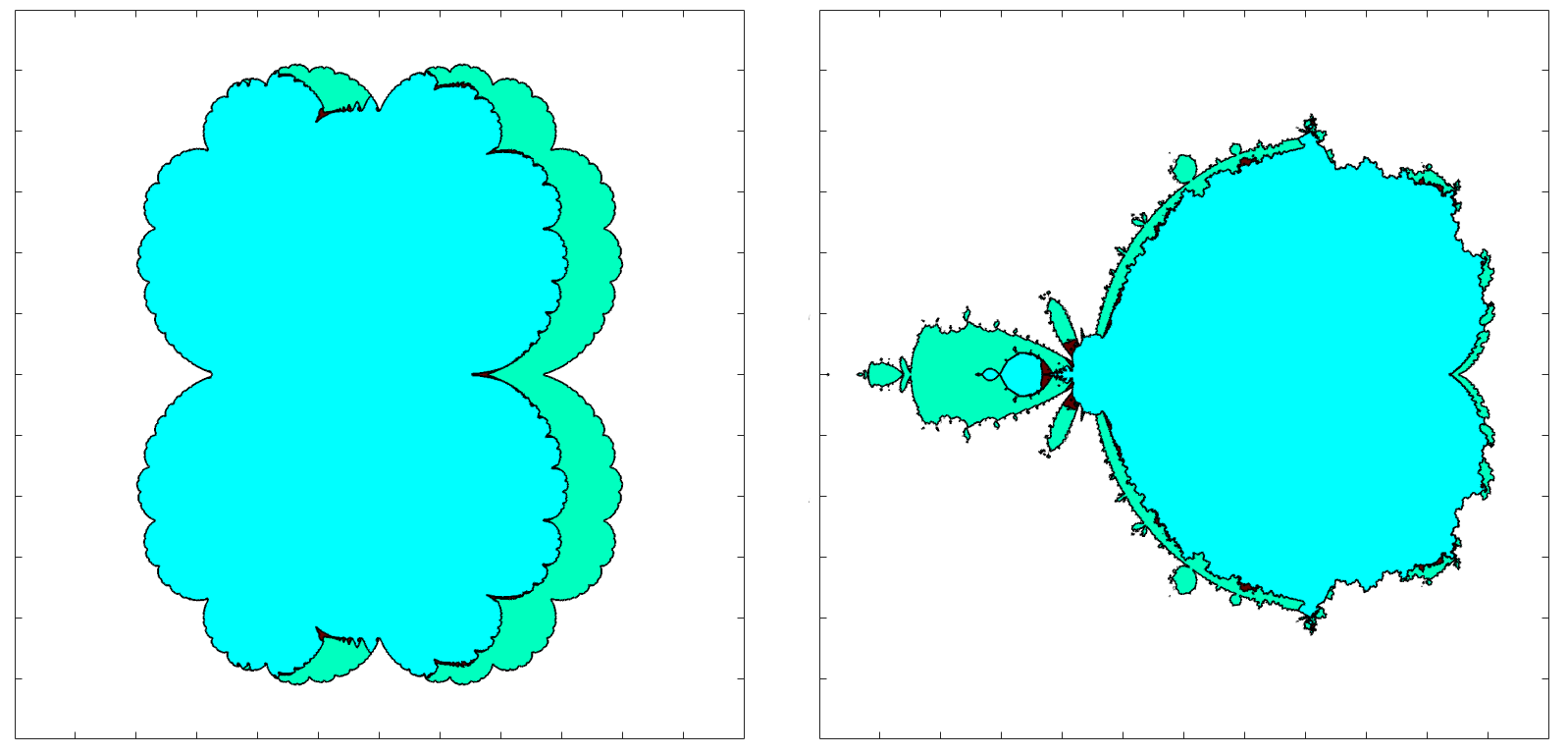}
\end{center}
    \caption{Slices of the Mandelbrot set for the template ${\bf s} =0111\hdots$ using two quadratic polynomials. The parameter is $({\bf c}, {\bf d}) = ((c_0, c_1), (d_0, d_1)) = ((\tfrac{1}{2} - \tfrac{1}{256}, \tfrac{1}{4} + \tfrac{1}{256}),(2,2))$ 
    and slices are taken in the $c_0$-direction on the left and in the $c_1$-direction on the right. The Mandelbrot slices for the limit template are shown in green while the approximate slices using periodic approximations to the limit template of periods $20$ and $200$ are shown in light blue and dark brown, respectively. 
The pictures were made using a maximum of $400$ iterations 
 and a resolution of $1200 \times 1200$. The plotting ranges are $[-1.25,1.25]\times[-1.25,1.25]$ for the $c_0$ panel and $[-1.5,0.5]\times[-1,1]$ for the $c_1$ panel.}
    \label{whole_set_hole}
\end{figure}

\begin{thm}
Let $D=d=2$, let ${\bf s}=(011111\ldots\ldots)$, and let $\{{\bf s}^N\}_{N=1}^\infty$ be any sequence of templates in $\{ 0,1 \}^\infty$ such that $\ds \lim_{N \to \infty} \| {\bf s}^N - {\bf s} \| =0$, and such that ${\bf s}^N \neq {\bf s}$ for infinitely many $N$. Then 
$$d({\cal M}^{\bf s},{\cal M}^{{\bf s}^N}) \not\to 0 \quad\mbox{as} \quad N \to \infty.$$
\label{counterexample}

\end{thm}
\noindent {\bf Remark.} The key to the proof of this result is the asymmetry of the sequence ${\bf s}$ above which passes from $0$ to $1$ but not from $1$ to $0$. On the other hand, the conditions on the approximating sequences ${\bf s}^N$ \emph{force} these sequences to make a transition from $1$ to $0$ for infinitely many such sequences. Using this, we can make the dynamics sufficiently different (at this point, the reader might like to consult Figure \ref{MappingsPic} below) to ensure that all critical points for template iteration using ${\bf s}$ with parameters $((c_0, c_1), (2,2))$ have bounded orbits while at the same time we can ensure that, for a reasonably small neighbourhood of these parameters, there must be a critical point whose orbit escapes if one uses the templates ${\bf s}^N$ (again for infinitely many such sequences). Before we can proceed with the proof of this result, we first need a small lemma about stability under perturbation for the basin of infinity. Since we are now in the situation of constant degree $2$, for convenience we will suppress the degree vectors in the parameters, referring to $(c_0, c_1)$ instead of $((c_0, c_1), (2,2))$ as above, using the product metric on ${\mathbb C}^2$ instead of ${\mathbb C}^2 \times \{2\}^2$ for our parameter neighbourhoods, and identifying our parameter spaces with subsets of ${\mathbb C}^2$ instead of ${\mathbb C}^2 \times \{2\}^2$.

\begin{lemma}
Let $0 < \varepsilon_0, \varepsilon_1 \leq \tfrac{1}{8}$ and let  $c_0=\tfrac{1}{2}-\varepsilon_0$, $c_1 = \tfrac{1}{4} - \varepsilon_1$. Then there exist $\eta >0$ and $\delta>0$ depending only on $\varepsilon_0$ and $\varepsilon_1$ such that, if ${\bf s} \in \{ 0,1 \}^\infty$ is any arbitrary template,  $d((c_0,c_1),(\tilde{c}_0,\tilde{c}_1))<\eta$, and $\{\tilde{P}_m^{\bf s} \}_{m=1}^\infty$ is the polynomial sequence constructed from the pair $(\tilde{c}_0,\tilde{c}_1)$ (and degrees $2$) according to the template ${\bf s}$, then
$${\mathrm {D}}(x,\delta) \subset \tilde{A}^{\bf s}_\infty, \text{ for all } x \geq \beta + \sqrt{\varepsilon_1},$$

\noindent {where $\beta := \tfrac{1}{2} + \sqrt{\epsilon_1}$  is the repelling fixed point for $f_{c_1}$} and $\tilde{A}^{\bf s}_\infty$ is the basin of infinity for $\{\tilde{P}_m^{\bf s} \}_{m=1}^\infty$. 
\label{lemma_from_hell}
\end{lemma}

\noindent \emph{\textbf{Proof of Lemma \ref{lemma_from_hell}}} Let ${\bf s} \in  \{ 0,1 \}^\infty$ be an arbitrary template and let ${\cal J}_{c_0}$, ${\cal J}_{c_1}$ and ${\cal K}_{c_0}$, ${\cal K}_{c_1}$ denote the Julia and filled Julia sets respectively for $f_{c_0}$, $f_{c_1}$ (again the reader is referred to Figure \ref{MappingsPic} for pictures of these Julia sets and how they are used in our proof). As mentioned above, $\beta$ is a repelling fixed point for $f_{c_1}$ and one checks that ${\cal K}_{c_1} \cap \mathbb{R} = [-\beta,\beta]$. Recall also that, provided $(\tilde{c}_0, \tilde{c}_1) \in \overline {\mathrm {B}}((0,0),2)$, then $R=2$ is an escape radius for any polynomial sequence $\{\tilde{P}_m^{\bf s} \}_{m=1}^\infty$ as above. Clearly, there exists $N_0 = N_0(\varepsilon_1)$ such that:
$$f_{c_1}^{\circ N_0} (\beta + \sqrt{\varepsilon_1}) >3.$$

\noindent Since $c_0 \ge c_1$ and $f_{c_0}$ and $f_{c_1}$ are both increasing real-valued functions on $(0,\infty)$, this implies that
\begin{equation}
Q^{\bf s}_{N_0}(x) >3, \text{ for all } x \,\ge \beta + \sqrt{\varepsilon_1}
\label{QN_0}
\end{equation}
\noindent where $Q^{\bf s}_{N_0}$ denotes the finite composition generated by the pair $(c_0, c_1)$ and the first $N_0$ entries of the template ${\bf s}$. We next proceed to finish the proof according to two cases depending on the value of the quantity $x$.\\

{\bf Case 1: $\beta+\sqrt{\varepsilon_1} \le x \le 3$.} Consider the collection of compositions $\{ \tilde{Q}^{\bf u}_{N_0}(z) \}$ where ${\bf u}$ is any template, and the polynomial sequences $\{\tilde{P}_m^{\bf u} \}_{m=1}^\infty$ (which give rise to the compositions $\{ \tilde{Q}^{\bf u}_{N_0}(z) \}$) are generated according to the template ${\bf u}$ and the pair $(\tilde{c}_0,\tilde{c}_1) \in {\mathbb C}^2$ (using degree $2$ for each polynomial $\tilde{P}_m^{\bf u}$). Each member of this collection is jointly differentiable in $z$, $\tilde c_0$, and $\tilde c_1$, and the number of possibilities for the first $N_0$ entries of the template ${\bf u}$ is finite (actually $2^{N_0}$). In addition, the degrees of these polynomials are all $2^{N_0}$ and their coefficients (in $z$) are fixed polynomials in $\tilde c_0$, $\tilde c_1$ (of degree at most $2^{N_0-1}$ in each variable). It follows (somewhat) as in the proof of Theorem \ref{usc} that 
the partial derivatives of these functions in $z$, $\tilde c_0$, $\tilde c_1$  are 
uniformly bounded on any bounded subset of $\mathbb{C}^3$. From this we can deduce that there exist $\eta,\delta>0$ (depending only on $\varepsilon_0$ and $\varepsilon_1$) such that, if $(\tilde{c}_0,\tilde{c}_1) \in {\mathrm {B}}((c_0, c_1), \eta)$, $\beta + \sqrt{\epsilon_1} \le x \le 3$, and $|z-x|<\delta$, then $(\tilde{c}_0,\tilde{c}_1) \in \overline {\mathrm {B}}((0,0), 2)$ and
$$|\tilde{Q}^{\bf u}_{N_0}(z) - Q^{\bf u}_{N_0}(x)|<1.$$

\noindent In conjunction with~\eqref{QN_0}, this implies (by  the reverse triangle inequality) that
$$\tilde{Q}^{\bf s}_{N_0}(\mathrm{D}(x,\delta)) \subset \mathbb{C} \setminus \overline{\mathrm{D}}(0,2).$$
Since $(\tilde{c}_0,\tilde{c}_1) \in \overline {\mathrm {B}}((0,0), 2)$, $R=2$ is an escape radius for any template sequence of quadratic polynomials $\{\tilde{P}_m^{\bf s} \}_{m=1}^\infty$ generated using $(\tilde{c}_0,\tilde{c}_1)$ and ${\bf s}$ as above. It therefore follows that $\mathrm{D}(x,\delta) \subset \tilde A^{\bf s}_\infty$.\\

{\bf Case 2:} $x >3$.
To conclude the proof: if $x > 3$, make $\delta$ smaller if needed so that $\delta \le 1$ (which still leaves the dependence on $\varepsilon_0$ and $\varepsilon_1$ unaltered). Then, since we still have $(\tilde{c}_0,\tilde{c}_1) \in \overline {\mathrm {B}}((0,0), 2)$, it follows that $R=2$ is an escape radius for $\{\tilde{P}_m^{\bf s} \}_{m=1}^\infty$.
We then automatically have ${\mathrm {D}}(x,\delta) \subset \tilde{A}^{\bf s}_\infty$ in this case also, which completes the proof. 
\hfill \qed

\vspace{5mm}
\noindent \emph{\textbf{Proof of Theorem ~\ref{counterexample}.}} Clearly the result follows if we show that, for any template sequence $\{{\bf s}^N\}_{N=1}^\infty$ such that ${\bf s}^N \ne {\bf s}$ for every $N>1$, we have $\limsup_N d({\cal M}^{\bf s},{\cal M}^{{\bf S}^N})>0$. We therefore start (without loss of generality) by assuming this is the case and that ${\bf s}^N \ne {\bf s}$ for every $N \ge 1$.

\afterpage{\clearpage}
\begin{sidewaysfigure}[htb]
  
  \hspace{-.075cm} \includegraphics[width=9cm]{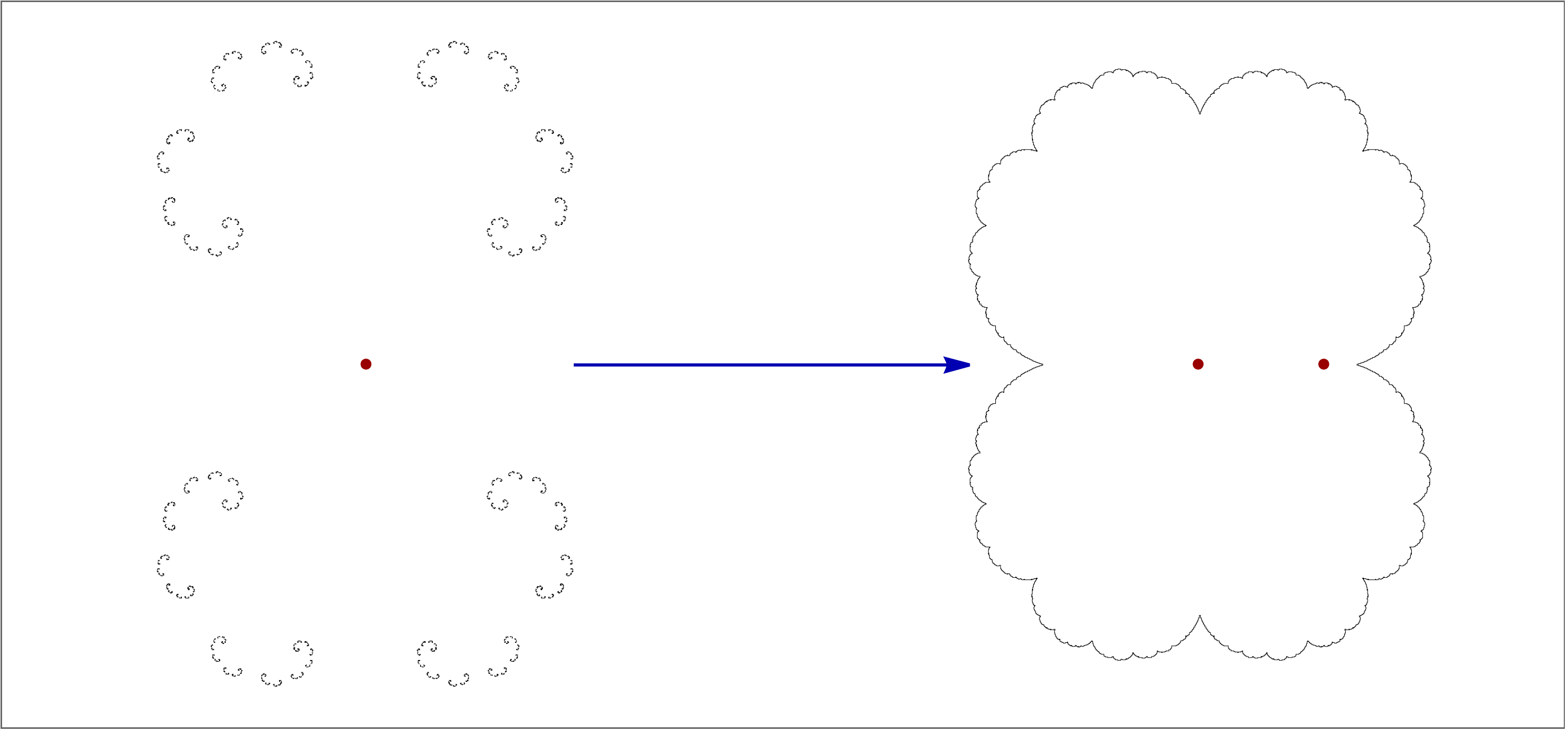}\\
   
  \vspace{2.1cm}
  \hspace{.05cm}\includegraphics[width=20cm]{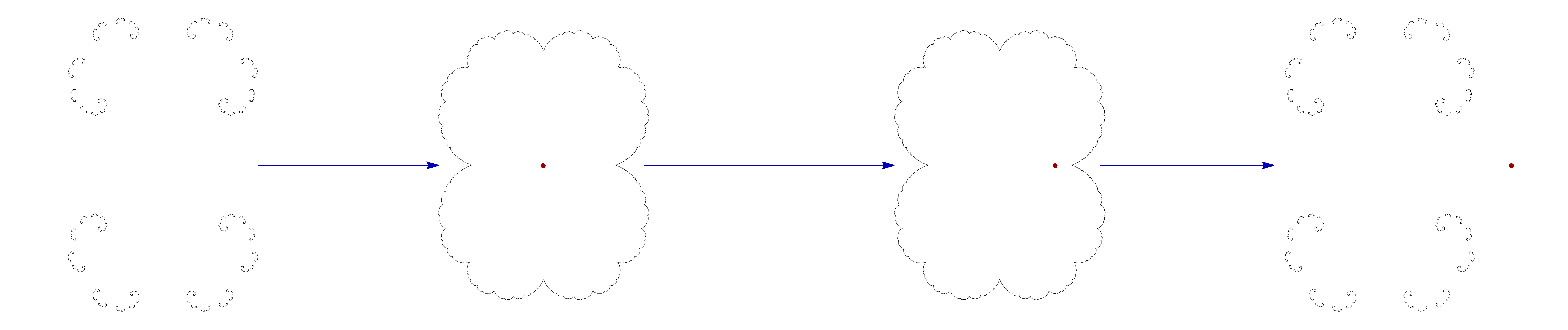}

  \vspace{.9cm}
  \caption{\footnotesize Schematic for the proof of Theorem~\ref{counterexample}. The top set of pictures illustrates the behaviour of the sequence$\{P_m^{\bf s} \}_{m=1}^\infty$ of quadratics generated by the pair $(c_0,c_1)$ and the template ${\bf s}$. Here the critical point $0$ at time $0$ is mapped inside ${\cal K}_1$ while all subsequent critical points also lie inside this filled Julia set and thus do not escape. In the bottom set of pictures we have the behaviour of the sequence of quadratics $\{\tilde{P}_m^{{\bf s}^N} \}_{m=1}^\infty$ generated by the pair $(c_0, c_1)$ and the template ${\bf s}^N$. In this case the critical point $0$ at time $1$ remains inside ${\cal K}_1$ for $N_1$ iterations. However, the next polynomial in the template sequence is $P_{N_1+2}^{{\bf s}^N} = f_{c_0}$ which maps the iterate of this critical point to the basin of infinity so that it escapes. Much of the proof is devoted to showing this behaviour is stable under perturbation and that there is a lower bound for the size of the perturbation which depends only on the constants $\epsilon_0$, $\epsilon_1$.}  \label{MappingsPic}
    \unitlength1cm
\begin{picture}(0.01,0.01)
\put(0.2,16.0){Limit Template ${\bf s} = 01111111\ldots\ldots$}
\put(0.2,9){Approximating Template ${\bf s}^N = 01111111\ldots\ldots1110******$}
\put(1.6,10.5){\footnotesize Stage $0$}
\put(1.6,3.6){\footnotesize Stage $0$}
\put(6.35,10.5){\footnotesize Stage $1$}
\put(6.35,3.6){\footnotesize Stage $1$}   
\put(12.25,3.6){\footnotesize Stage $N_1+1$}
\put(16.75,3.6){\footnotesize Stage $N_1+2$}   
\put(3.85,13.45){\footnotesize \textcolor[rgb]{0,0,0.7}{$P_1^{\bf s} = f_{c_0}$}}          
\put(3.8,6.65){\footnotesize \textcolor[rgb]{0,0,0.7}{$P_1^{{\bf s}^N} = f_{c_0}$}}    
\put(8.8,6.65){\footnotesize \textcolor[rgb]{0,0,0.7}{$Q^{{\bf s}^N}_{1,N_1+1} = f_{c_1}^{\circ N_1}$}}  
\put(14.55,6.65){\footnotesize \textcolor[rgb]{0,0,0.7}{$P_{N_1+2}^{{\bf s}^N} = f_{c_0}$}}  
\put(3.3,14.9){\footnotesize \textcolor[rgb]{0,0.5,0}{${\cal J}_{c_0}$}}     
\put(8.2,14.9){\footnotesize \textcolor[rgb]{0,0.5,0}{${\cal J}_{c_1}$}}     
\put(3.3,8){\footnotesize \textcolor[rgb]{0,0.5,0}{${\cal J}_{c_0}$}} 
\put(8.2,8){\footnotesize \textcolor[rgb]{0,0.5,0}{${\cal J}_{c_1}$}} 
\put(14,8){\footnotesize \textcolor[rgb]{0,0.5,0}{${\cal J}_{c_1}$}} 
\put(18.9,8){\footnotesize \textcolor[rgb]{0,0.5,0}{${\cal J}_{c_0}$}} 
 \end{picture}
\end{sidewaysfigure}
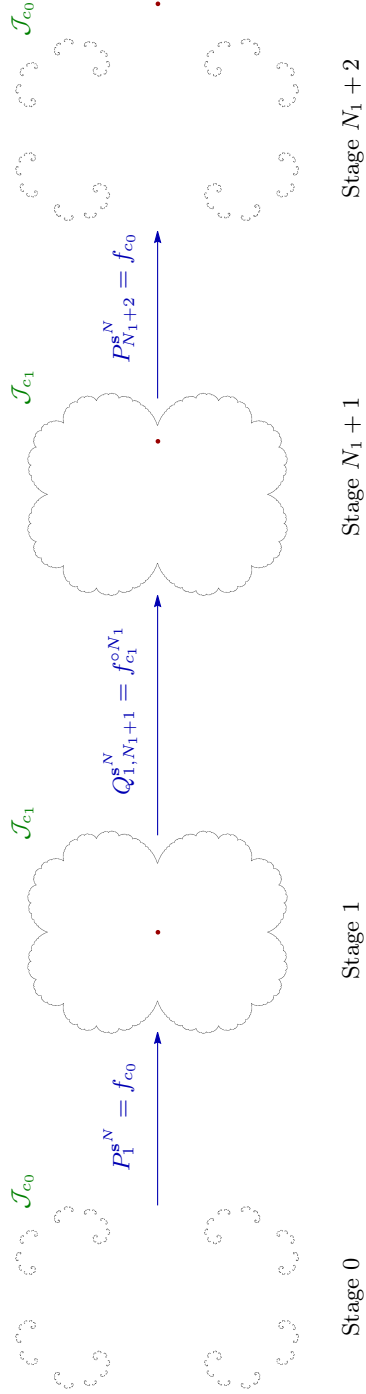

\newpage

As before, let $0<\varepsilon_0, \varepsilon_1 \le \tfrac{1}{8}$ and set $c_0 = \frac{1}{2} - \varepsilon_0$, $c_1 = \frac{1}{4} -\varepsilon_1$.  Recall that, as a single map under iteration, $f_{c_1}(z) = z^2+\frac{1}{4} - \varepsilon_1$ has two fixed points: $\alpha := \frac{1}{2}-\sqrt{\varepsilon_1}$ (which is attracting) and $\beta := \frac{1}{2}+\sqrt{\varepsilon_1}$ (which is repelling). As above, the filed Julia set ${\cal K}_{c_1}$ of $f_{c_1}$ still satisfies ${\cal K}_{c_1} \cap \mathbb{R} = [-\beta,\beta]$. Then 
$$f_{c_0}(0) = \frac{1}{2} - \varepsilon_0 \in [-\beta,\beta] \subset {\cal K}_{c_1}.$$

\noindent Thus, $f_{c_0}(0)$ has bounded orbit under $f_{c_1}$. Since $0$ also has bounded orbit under $f_{c_1}$, all critical points of the sequence $\{P_m^{\bf s} \}_{m=1}^\infty$ generated by the pair $(c_0,c_1)$ and the template ${\bf s}$ (as well as the degree pair $(2,2)$) have bounded orbit and so 
$$(c_0,c_1) \in {\cal M}^{\bf s}.$$

Now take $\varepsilon_0$ and $\varepsilon_1$ sufficiently small such that:

\begin{equation}
3\sqrt{\varepsilon_1} - \varepsilon_1 + \varepsilon_0 < \frac{1}{4}.
\label{eps_ineq}
\end{equation}

\noindent Then
\begin{equation}f_{c_0}(\alpha) = \alpha^2 + \frac{1}{2} -\varepsilon_0 > \beta +\sqrt{\varepsilon_1}.
\label{totherightofbeta}
\end{equation}

Let $\eta,\delta>0$ (depending on $\varepsilon_0$, $\varepsilon_1$) be as in Lemma \ref{lemma_from_hell}, let $m_0 \ge0$ be arbitrary, and let ${\bf u} = \{u_m\}_{m=1}^\infty$ be any template satisfying $u_{m_0+1}=0$. Hence, using \eqref{totherightofbeta}, Lemma \ref{completelyinvariant}, and Lemma \ref{lemma_from_hell}, if $d((c_0, c_1), (\tilde c_0, \tilde c_1))<\eta$, then 
\begin{equation}
 \mathrm{D}(f_{c_0}(\alpha),\delta) \subset \tilde{A}^{\bf u}_{\infty,m_0+1},
\label{basin_of_infinity}
\end{equation}

\noindent where the polynomial sequence $\{\tilde{P}_m^{\bf u} \}_{m=1}^\infty$ is generated according to ${\bf u}$ and $(\tilde c_0, \tilde c_1)$ (as well as the degree pair $(2,2)$), while $\tilde{A}^{\bf u}_{\infty,m_0+1}$ represents the basin of $\infty$ at time $m_0+1$ for this polynomial sequence.

Since $f_{\tilde c_0}$ is continuous, it follows that there exists a $\tau_0>0$ (also depending on $\varepsilon_0$ and $\varepsilon_1$), such that:
\begin{equation}
f_{c_0}(\mathrm{D}(\alpha,\tau_0)) \subset \mathrm{D}(f_{c_0}(\alpha),\tfrac{\delta}{2}).
\label{disc_inclusion}
\end{equation}

Let $\eta'= \min(\eta,\delta/2)$ (which also depends on only $\varepsilon_0$ and $\varepsilon_1$). It follows that, when $|\tilde{c}_0-c_0| < \eta'$ (so that $|f_{\tilde c_0}(z) - f_{c_0}(z)| = |\tilde c_0 - c_0| < \eta' \le \tfrac{\delta}{2}$ for any $z \in {\mathbb C}$), then, by ~\eqref{basin_of_infinity}, ~\eqref{disc_inclusion}
\begin{equation}
\tilde{Q}^{\bf u}_{m_0,m_0+1}(\mathrm{D}(\alpha,\tau_0)) = \tilde P^{\bf u}_{m_0+1}(\mathrm{D}(\alpha,\tau_0)) = \tilde{f}_{c_0}(\mathrm{D}(\alpha,\tau_0)) \subset \mathrm{D}(f_{c_0}(\alpha),\delta) \subset \tilde{A}^{\bf u}_{\infty,m_0+1}.
\label{mapstobasinofinfinity}
\end{equation}

Recall that $f_{c_1}$ has an attracting fixed point at $\alpha = 1/2 -\sqrt{\varepsilon_1}$, with multiplier $2\alpha = 1-2\sqrt{\varepsilon_1}$. Since this condition is clearly stable under perturbation, we can make $\eta'$ smaller if needed (still only depending on $\varepsilon_0$ and $\varepsilon_1$), such that, if $|\tilde{c}_1-c_1|<\eta'$, then $f_{\tilde{c}_1}$ has a fixed point $\tilde{\alpha}$ such that
\begin{equation}
|\tilde{\alpha}-\alpha| < \tau_0/2
\label{**}
\end{equation}

\noindent
while we can at the same time ensure that the multiplier $2\tilde{\alpha}$ of this fixed point has absolute value strictly less than $1-\sqrt{\varepsilon_1}$. The multiplier of $\tilde{\alpha}$ is thus uniformly bounded below away from 1 in absolute value so that the fixed points $\tilde{\alpha}$ are uniformly attracting while the immediate basin of each of them must contain the critical point $0$ (see e.g. \cite{CG} Chapter III Theorem 2.2). An easy argument using relative compactness of $\mathrm D(c_1, \eta')$) shows that there then exists $N_0>0$ (depending only on $\varepsilon_0$ and $\varepsilon_1$), such that, in view of \eqref{**}:
\begin{equation}
f_{\tilde{c}_1}^{\circ N}(0) \in \mathrm{D}(\tilde{\alpha},\tau_0/2) \subset \mathrm{D}(\alpha,\tau_0), \text{ for all } N \geq N_0.
\label{***}
\end{equation}

\begin{figure}[H]
\label{ZoomedSets}
\begin{center}
    \includegraphics[width=1\textwidth]{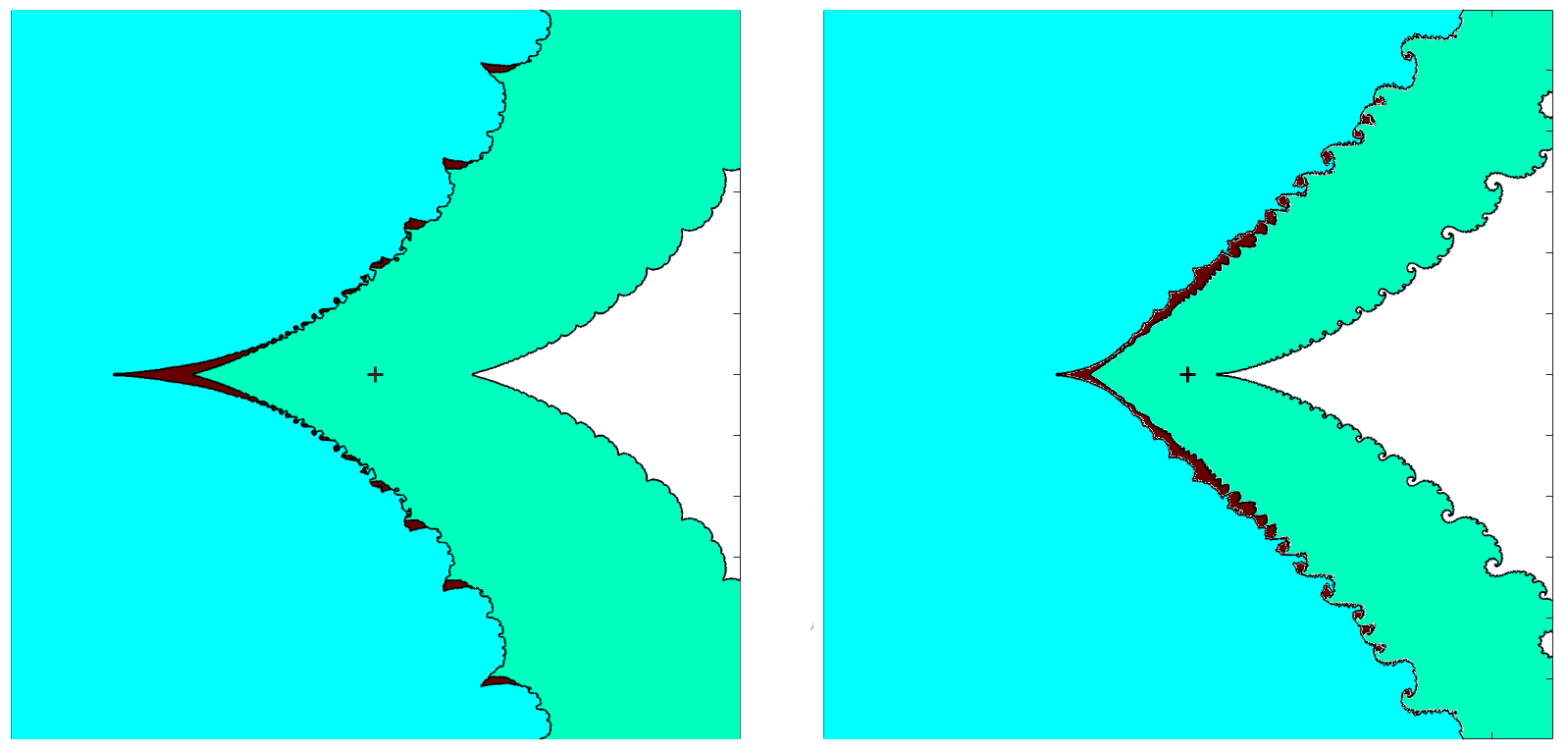}
\end{center}
    \caption{Zooms about $(c_0, c_1)=(\tfrac{1}{2} - \tfrac{1}{256}, \tfrac{1}{4} + \tfrac{1}{256})$ (with degree vector $(2,2)$) of the pictures in Figure \ref{whole_set_hole} illustrating the lack of lower semicontinuity where $d({\cal M}^{\bf s},{\cal M}^{{\bf s}^{N}})$ does not tend to zero as $N$ tends to infinity. Here we have $\varepsilon_0 = \varepsilon_1 = \tfrac{1}{256}$ which satisfies condition \eqref{eps_ineq} in the proof of Theorem \ref{counterexample} above. Again, slices in the $c_0$-direction are on the left and in the $c_1$-direction on the right. The cross marks the point of intersection of these two complex lines at $(c_0,c_1)$. As before, we have a maximum of $400$  iterations, and a resolution of $1200 \times 1200$. As with Figure \ref{whole_set_hole}, the Mandelbrot slices for the limit sequence are in green, with those for the periodic approximations of periods $20$ and $200$ in light blue and and dark brown respectively. The plotting ranges are $[c_0-0.25,c_0+0.25]\times[-0.25,0.25]$ for the $c_0$ panel and $[c_1-0.05,c_1+0.05]\times[-0.05,0.05]$ for the $c_1$ panel.}
    \label{whole_set_zoom}
\end{figure}

%%%%%%%%%%%%%%%%%%%%

\noindent Now we look at the sequence of templates $\{ {\bf s}^N \}_{N=1}^\infty$. Since $\| {\bf s}^N - {\bf s}\| \to 0$ and since we have assumed that ${\bf s}^N \neq {\bf s}$ for all $N$, it follows that there exists $N_1$ (which thus tends to infinity as $N$ tends to infinity) such that 

\begin{itemize}
\item $\tilde{Q}^{{\bf s}^N}_{0,1} = \tilde{P}^{{\bf s}^N}_1 = f_{\tilde{c}_0}$,

\item $\tilde{Q}^{{\bf s}^N}_{1,N_1+1} = f_{\tilde{c}_1}^{\circ N_1}$, and

\item $\tilde{Q}^{{\bf s}^N}_{N_1+1,N_1+2} = \tilde{P}^{{\bf s}^N}_{N_1+2} = f_{\tilde{c}_0}$.
\end{itemize}

\noindent 
In other words, the first member of the sequence $\{\tilde{P}^{{\bf s}^N} \}_{m=1}^\infty$ is $f_{\tilde{c}_0}$, the next $N_1$ members are $f_{\tilde{c}_1}$, and the $(N_1+2)$-nd member is again $f_{\tilde{c}_0}$.
Then, by \eqref{mapstobasinofinfinity} and \eqref{***} (where we let $m_0 = N_1+1$ in ~\eqref{mapstobasinofinfinity}), it follows that, if we choose $N$ sufficiently large so that we also have $N_1 \ge N_0$, then, %{\color{orange}using Lemma \ref{completelyinvariant} on complete invariance,} 
$$\tilde{Q}_{1,N_1+2}^{{\bf s}^N}(0) \in \tilde{A}_{\infty,N_1+2}^{{\bf s}^N}.$$ 

It follows that, for $N$ sufficiently large, if $|\tilde{c}_0-c_0| < \eta'$, then the critical point zero at time one escapes for any template in this sequence, that is:
$$(\tilde{c}_0,\tilde{c}_1) \notin {\cal M}^{{\bf s}^N}$$

\noindent Hence 
$$\mathrm{B}((c_0,c_1),\eta') \subset \mathbb{C}^2 \setminus {\cal M}^{{\bf s}^N}$$

\noindent and thus
$$d((c_0,c_1),{\cal M}^{{\bf s}^N}) \geq \eta'$$

\noindent (see Figure \ref{whole_set_zoom} above) so that 
$$\limsup _{N \to \infty} d({\cal M}^{\bf s},{\cal M}^{{\bf s}^{N}}) >0$$

%$$d({\cal M}^{\bf s},{\cal M}^{{\bf s}^N}) \not \to 0, \text{ as } N \to \infty$$

\noindent as required, from which the result follows.
\hfill \qed

\bibliographystyle{plain}
\bibliography{References3}

\end{document}